\documentclass[11pt, preprint]{amsart}
\usepackage{amsfonts,amssymb,amsmath,amsthm,amscd,mathtools,multicol,tikz, tikz-cd, caption, enumerate, mathrsfs, geometry,tgpagella}
\usepackage[foot]{amsaddr}
\usepackage[pagebackref=true, colorlinks, linkcolor=blue, citecolor=red]{hyperref}
\usepackage{fullpage}
\usepackage[all]{xy}
\makeatletter
\def\blfootnote{\gdef\@thefnmark{}\@footnotetext}
\makeatother

\makeindex
\newtheorem{theorem}{Theorem}[section]
\newtheorem{proposition}[theorem]{Proposition}
\newtheorem*{theorem*}{Theorem}

\newtheorem{lemma}[theorem]{Lemma}
\theoremstyle{remark}
\newtheorem{remark}[theorem]{Remark}

\numberwithin{equation}{section}

\title{elementary first integrals  of  integral differential systems}
\author{Chitrarekha Sahu and Varadharaj R. Srinivasan}
\address{Indian Institute of Science Education and Research (IISER) Mohali Sector 81, S.A.S. Nagar, Knowledge City, Punjab 140306, India.}
\email{\url{chitrarekhasahu97@gmail.com},  \url{ravisri@iisermohali.ac.in}}	
\begin{document}
\begin{abstract}    
For a field $F$ of characteristic zero with a derivation $\delta$, we provide a necessary and sufficient condition for a system of differential equations \begin{equation*}\delta y_1=f_0, \quad\delta y_2=f_1,\quad\dots\quad,\delta y_n=f_{n-1},\end{equation*} where $f_0\in F$, $f_1\in F[y_1]$, $\dots$, $f_{n-1}\in F[y_1,\dots,y_{n-1}]$, to have elementary first integrals.
\end{abstract}
\maketitle
%\tableofcontents

%%%%%%%%%%%%%%%%%%%%%%%%%%%%%%%%%%%%%%%%%%%%%%%%%%%%%%%%%%%%%%%%%%%%%%%%%%%%%%%%%%%%%%%%%%%%%%%%%%%%%%%%%%%%%%%%%%%%%%%%%%%%%%%%%%%%%%%%%%%%%%%%%%%%%%%%
%%%%%%%%%%%%%%%%%%%%%%%%%%%%%%%%%%%%%%%%%%%%%%%%%%%%%%%%%%%%%%%%%%%%%%%%%%%%%%%%%%%%%%%%%%%%%%%%%%%%%%%%%%%%%%%%%%%%%%%%%%%%%%%%%%%%%%%%%%%%%%%%%%%%%%%%	
\section{Introduction}\label{introduction}

\subsection{Definitions and terminologies} A field (respectively, a ring) with a finite set of commuting derivations $\Delta$ 
is called a \emph{$\Delta$-field} (respectively, a \emph{$\Delta$-ring}). If $\Delta$ 
consists of a single derivation $\delta$, then the field (respectively, the ring) 
is called a \emph{$\delta$-field} (respectively, \emph{$\delta$-ring}). Throughout this article, $F$ denotes a $\delta$-field of characteristic zero having an algebraically closed \emph{field of constants} $F^\delta:=\{x\in F \ | \ \delta x=0\}$. A $\delta$-field $E$ is said to be a \emph{$\delta$-extension} of $F$ (or equivalently, $F$ is said to be a \emph{$\delta$-subfield} of a $\delta$-field $E$) if $E$ is a field extension of $F$ and the derivation on $E$ restricts to the derivation on $F$.

Let $F(Y):=F(y_1,\dots,y_n)$ be the rational function field in $n$ variables over $F$ having a set of derivations $\Delta = \{ \delta_0, \delta_{1}, \dots, \delta_{n} \}$ on $F(Y)$ such that for each $i,j\in\{1,\dots, n\}$
\begin{enumerate}[(1)]
    \item $\delta_0(x)=\delta(x)$ for all $x\in F$  and $\delta_0y_j = 0$.
    \item $\delta_{i}x=0$ for all $x\in F,$ $\delta_{i}y_j=0$ if $i\neq j$ and $\delta_{i}y_i=1$.
\end{enumerate} 
It is easily seen that the set of constants $F(Y)^\Delta:=\{x\in F(Y)\ | \ \delta_i x=0\ \text{for all } \delta_i\in \Delta\}$ forms a subfield of $F(Y)$. The elements of $F(Y)^\Delta$ (respectively $F^\delta$) are sometimes called \emph{$\Delta$-constants} (respectively \emph{$\delta$-constants}). A field extension $E$ of $F(Y)$ is said to be an \emph{elementary $\Delta$-extension} of $F(Y)$ if each $\delta_i\in \Delta$ can be extended to a derivation of $E$ such that  $E^\Delta=F(Y)^{\Delta}$ and that there are elements $t_1,\dots,t_m\in E$ such that $E=F(Y)(t_1,\dots,t_m)$ and for each $j=1,2,\dots,m,$ there exists an element $a_j\in F(Y)(t_1,\dots,t_{j-1})$ with the property that either $t_j$ is algebraic over $F(Y)(t_1,\dots,t_{j-1})$ or $\delta_i t_j=(\delta_i a_j)t_j$ for all $i=0,1,2,\dots,n$ or $a_j$ is nonzero and  $\delta_i t_j=\frac{\delta_i a_j}{a_j}$ for all $i=0,1,2,\dots,n$.

Consider the derivation $$\delta:=\delta_0+f_1\delta_{1}+\cdots+f_n\delta_{n}$$ on $F(Y)$,  where $f_1,\dots,f_n\in F(Y)$.
An element $u\in F(Y)$ is called  an \emph{elementary $\delta$-integral} of $F$  if $F(Y)^\delta=F^{\delta}$ and there are  an element $v_0\in F,$  nonzero elements $v_1,\dots, v_n\in F$ and  constants $c_1,\dots, c_n\in F^{\delta}$  such that  $\delta u\in F$ and that $$\delta u=c_1\frac{\delta v_1}{v_1}+\cdots+c_n\frac{\delta v_n}{v_n}+\delta v_0.$$

Let $E$ be an elementary $\Delta$-extension of $F(Y).$   An element $e\in E^\delta$ and $e\notin F^{\delta}$ is called an \emph{elementary first $\delta$-integral} of the system   \begin{equation}\label{D-system}\delta y_1=f_1,\, \delta y_2=f_2, \,\dots,\,\delta y_n=f_n.\end{equation}

\subsection{Statement of the main theorem}For elements $f_0\in F,$ $f_1\in F[y_1],$ $\dots, f_{n-1}\in F[y_1,\dots,y_{n-1}],$ consider the derivation \begin{equation}\label{ADE-Derivation}
\delta := \delta_0+f_0 \delta_{1} + f_1\delta_{2}\cdots + f_{n-1}\delta_{n}
\end{equation}
and the first order system\begin{equation}\label{rationalsystem} \delta  y_1=f_0, \ \delta y_2=f_1,\cdots, \ \delta y_{n}=f_{n-1}.\end{equation}

\begin{theorem}\label{intromain} Suppose that  $F(Y)^{\delta}=F^{\delta}.$ Then, the system (\ref{rationalsystem}) admits an elementary first$\delta$-integral if and only if there is an elementary $\delta$-integral $u\in F(Y)$ of $F$ and that $u\notin F$.
\end{theorem}

\subsection{Some remarks on the main Theorem} 
Our theorem is an analogue of a result of Prelle and Singer \cite[Corollary 2]{MR704611}, wherein the case  $f_0,f_1,\dots, f_{n-1}\in F$ is handled (See Remark \ref{Singersresult}) and the proof of our Theorem  uses  \cite[Theorem]{MR704611}. In the language of Model Theory of Differential Fields, the above theorem can be rephrased as follows:

\begin{itemize}
    \item[] \emph{A weakly orthogonal generic type\footnote{The differential field generated by $F$ and a generic solution of the system \eqref{rationalsystem}  has $F^\delta$ as its field of constants, that is, $F(Y)^\delta=F^{\delta}$. Equivalently, the system \eqref{rationalsystem} has no rational first $\delta$-integrals.} described by \eqref{rationalsystem} admits an elementary first $\delta$-integral if and only if it admits an elementary $\delta$-integral.} 
\end{itemize}

We now describe a connection between system \eqref{rationalsystem} and a specific class of Picard-Vessiot $\delta$-extensions of $F$.  Let  $E$ be a Picard-Vessiot $\delta$-extension  of $F$  such that the differential Galois group is isomorphic to a unipotent algebraic group. Then every intermediate $\delta$-subfield $L$ is a solution field (in the sense of  Yves Andr\'{e} \cite{andre2014solution}) and there are $F$-algebraically independent elements $y_1,\dots, y_n\in L$ such that $L=F(y_1,\dots,y_n),$ $\delta y_1\in F$ and for each $i=2, 3,\dots,n,$  $\delta y_i\in F[y_1,\dots, y_{i-1}]$ (See \cite[Theorem 4.4]{KRS24}). Conversely, suppose that $F(Y)$ has a derivation $\delta$ as in \eqref{ADE-Derivation} and that $F(Y)^{\delta }$ $=F^\delta$. It is easily seen that the $\delta$-ring $F[y_1]$ consists of elements holonomic over $F$. In fact, $F[y_1]$ is the Picard-Vessiot $F$-algebra of the Picard-Vessiot $\delta$-extension $F(y_1)$ of $F$. Since $\delta y_2\in F[y_1]$, it follows that $y_2$, and consequently, every element of the ring $F[y_1,y_2]$, is holonomic over $F$. Iteratively, we see that  $F[y_1,\dots, y_n]$ consists of elements holonomic over $F$. Therefore, $F(Y)$ is a solution field. From \cite[Theorem 1.2.2]{andre2014solution}, $F(Y)$ can be embedded in a Picard-Vessiot $\delta$-extension $E$ of $F$.  Let $\mathscr G$ be the differential Galois group of $E$ over $F$ and $K$ be the compositum of all the images of $F(Y)$ under the elements of $\mathscr G$. Then, $K$ is readily seen to be a Picard-Vessiot $\delta$-extension of $F$. Let $S:=\{\sigma(y_i)\ | \ \sigma\in \mathscr G, 1\leq i\leq n\}$. Then $K=F(S)$ and we can find elements $x_1,\dots,x_m\in S$ such that $K=F(x_1,\dots, x_m)$, $\delta x_1\in F$ and $\delta x_i\in F(x_1,\dots, x_{i-1})$ for $i=2,\dots, n$. It is well-known that such a Picard-Vessiot extension $K$ of $F$ has a unipotent differential Galois group \cite[Sections 24 and 25]{kolchin1948algebraic}. 

If the $\Delta$-field $F(Y)$ contains an elementary $\delta$-integral of $F$ that is not in $F$ then, using  the proof of the our Theorem, one can  construct an elementary first $\delta$-integral. Through a simple example, we shall now illustrate this construction. Consider the rational function field $\mathbb C(x)$ over the complex numbers with the derivation $\delta =d/dx$ and the rational function field $\mathbb C(x)(y_1,y_2)$ with the set of derivations $\Delta=\{\delta_0, \delta_1, \delta_2\},$ where $\delta_1$ and $\delta_2$ are, as defined earlier, partial derivations with respect to $y_1$ and $y_2$ respectively.  Let $$\delta = \delta_0-\frac{y_2}{x}\delta_1-\frac{1}{1-x}\delta_2.$$ Then, it can be shown that $\mathbb C(x)(y_1,y_2)^{\delta }=\mathbb C$. In fact, the $\delta$-field $\mathbb C(x)(y_1,y_2)$ is isomorphic to the differential field generated by the dilogarithm $\int \frac{\log(1-x)}{x}dx$ over $\mathbb C(x).$  Let $E=\mathbb C(x)(y_1,y_2)(y)$, where $y$ is transcendental over $\mathbb C(x)(y_1, y_2)$ and extend the derivations of $\Delta$ to $E$ using the rule $$\delta_i y=\frac{\delta_i(1-x)}{1-x},\quad\text{for } i=0,1,2.$$
Then, clearly, $E$ is an elementary $\Delta$-extension of $\mathbb C(x)(y_1,y_2).$ Note that $\delta y=-1/(1-x),$ $\delta_1y=\delta_2y=0,$ $E^\Delta=\mathbb C(x)^\delta=\mathbb C$ and that $$\delta (y-y_2)=\frac{-1}{1-x}+\frac{1}{1-x}=0.$$ 

It is worthy to observe that every differential system of the form \eqref{rationalsystem} defined over $\mathbb C(x)$  admits either a rational first $\delta-$integral or an elementary first $\delta$-integral. Indeed, since $\delta y_1\in \mathbb C(x)$, either $\mathbb C(x)(y_1)$ contains a constant not in $\mathbb C$ or $y_1$ must be an elementary $\delta$-integral\footnote{There exist an element $G\in \mathbb C(x)$, complex numbers $\alpha_1,\dots,\alpha_m, c_1,\dots,c_m$ such that  $c_i\neq 0$ for at least one $i$, $\alpha_i\neq \alpha_j$ for all $i\neq j$ and that $y_1=G+\sum^m_{i=1}c_i\log(x-\alpha_i)$.} of $\mathbb C(x).$ In the latter case, our theorem  guarantees an elementary first $\delta$-integral. In the last section, we shall provide a family of differential systems that admit no elementary first $\delta $-integrals.
%%%%%%%%%%%%%%%%%%%%%%%%%%%%%%%%%%%%%%%%%%%%%%%%%%%%%%%%%%%%%%%%%%%%%%%%%
%%%%%%%%%%%%%%%%%%%%%%%%%%%%%%%%%%%%%%%%%%%%%%%%%%%%%%%%%%%%%%%%%%%%%%%%%
%%%%%%%%%%%%%%%%%%%%%%%%%%%%%%%%%%%%%%%%%%%%%%%%%%%%%%%%%%%%%%%%%%%%%%%%%
%%%%%%%%%%%%%%%%%%%%%%%%%%%%%%%%%%%%%%%%%%%%%%%%%%%%%%%%%%%%%%%%%%%%%%%%%
%%%%%%%%%%%%%%%%%%%%%%%%%%%%%%%%%%%%%%%%%%%%%%%%%%%%%%%%%%%%%%%%%%%%%%%%%
\section{Iterated and generalized iterated integrals}
%%%%%%%%%%%%%%%%%%%%%%%%%%%%%%%%%%%%%%%%%%%%%%%%%%%%%%%%%%%%%%%%%%%%%%%%%
%%%%%%%%%%%%%%%%%%%%%%%%%%%%%%%%%%%%%%%%%%%%%%%%%%%%%%%%%%%%%%%%%%%%%%%%%
%%%%%%%%%%%%%%%%%%%%%%%%%%%%%%%%%%%%%%%%%%%%%%%%%%%%%%%%%%%%%%%%%%%%%%%%%

In this section, we will only deal with fields equipped with a single derivation $\delta$. All $\delta$-extensions of $F$ are assumed to have $F^\delta$ as its field of constants. An element $\eta$ in a $\delta$-extension of $F$ is called a \emph{generalized iterated integral} of $F$ if there is a natural number $r$ and nonzero elements $a_1,\dots,a_r\in F$  such that 
\begin{equation}\label{gii-defn}a_1\delta (a_2\delta (\dots(a_r\delta (y))))=0.\end{equation}
A generalized iterated integral of $F$ is called an \emph{iterated integral} of $F$ if each $a_i$ can be chosen to be $1,$ that is, $\delta^r \eta\in F.$ An iterated integral is called an \emph{integral} of $F$ if $r=1,$ that is, $\delta \eta\in F.$

In the next proposition, it is shown that if $F$ has an element whose derivative is $1$ and $\eta\in E$ is an iterated integral of $F,$ then $F\langle \eta\rangle$, the $\delta$-field  generated over $F$ by $\eta$ and all its derivatives, is a field generated over $F$ by integrals of $F.$ Furthermore, these integrals can in turn be described in terms of $\eta$ and its derivatives.

    \begin{proposition}\label{integral_extension}
        Suppose that $F$ has an element $x$ such that $\delta  x=1.$ Let $\eta$ be an iterated integral of $F$
         with $\delta ^{r}\eta=f\in F$ and for $j=0,1,2,\dots,r-1,$ let
         \begin{equation}\label{w_idescription}w_{j+1}=\sum^j_{i=0}(-1)^i\frac{x^i}{i!}\delta^{r-(j+1-i)}\eta.\end{equation}
         Then, for each $j=0,1,2,\dots,r-1,$ \begin{equation}\label{derivativeofw_{j+1}}\delta w_{j+1}=\frac{(-1)^jx^j}{j!}f\in F\end{equation} and that  
         $F\langle \eta\rangle=F(w_1,\dots,w_{r}).$
    \end{proposition}
    
 \begin{proof} Taking derivative of $w_{j+1},$ we obtain \begin{align*}\delta w_{j+1}&=\sum^j_{i=1}(-1)^i\frac{x^{i-1}}{(i-1)!}\delta^{r-(j+1-i)}\eta+\sum^j_{i=0}(-1)^i\frac{x^i}{i!}\delta^{r-(j-i)}\eta\\ &=\sum^{j-1}_{i=0}(-1)^{i+1}\frac{x^{i}}{i!}\delta^{r-(j-i)}\eta+\sum^j_{i=0}(-1)^i\frac{x^i}{i!}\delta^{r-(j-i)}\eta\\ &= (-1)^j\frac{x^j}{j!}f.\end{align*}

  For $j=2,\dots,r,$ \begin{equation}\label{formulaw_j}w_j-\sum^{j-1}_{i=1}(-1)^i\frac{x^i}{i!}\delta^{r-(j-i)}\eta=\delta^{r-j}\eta.\end{equation}
Note that $w_1=\delta^{r-1}\eta,$ $w_2+xw_1=\delta^{r-2}\eta$ and recursively, we obtain for each $j=3,\dots,r,$ that $\delta ^{r-j}\eta\in F(w_1,\dots,w_j).$ Thus, $F\langle \eta\rangle=F(w_1,\dots,w_r).$\end{proof}

%\sum^{j-1}_{i=1}\frac{x^i}{i!}w_{j-i}
	For a later use, we also record the following lemma of Rosenlicht \cite{Ros68}.
	\begin{lemma}\label{Rosenlicht-lemma}
		Let $\eta\notin F$ be an element in a  $\delta$-extension of $F$ and let $\delta \eta=\alpha \in F$. Let $v_1,\dots,v_r\in {F(\eta)^*},v\in {F(\eta)}$ and $\mathbb Q$-linearly independent constants $c_1,\dots,c_r\in F^{\delta}$ such that 
		$$\sum_{i=1}^{r}c_i\frac{\delta v_i}{v_i}+\delta v=P(\eta)\in F[\eta]$$ 
		is of degree $s$ in $\eta$ then each $v_i\in F$ and $v\in F[\eta]$ is a polynomial of degree $\leq s+1$.
		%and the coefficient of $\eta^{s+1}$ in $\sum_{i=1}^{r}c_i\frac{v_i'}{v_i}+v'$ is in $C$. 
	\end{lemma}
    \begin{proof}
 From \cite[Lemma]{Ros68} we get that each $v_i\in F$ and $v\in F[\eta].$ Now, it follows from \cite[Proposition 2.2 (c)]{KS19} that degree of $v$ is  $s$ or $s+1.$ 
\end{proof}

If $\eta$ is a generalized iterated integral of $F$ then it can be shown (See \cite[Remark 2.7]{SSS25}) that the $\delta$-field generated by $\eta$ and $F$ is of the form $F(w_1,\dots, w_n),$ where $w_1,\dots, w_n$ are $F$-algebraically independent over $F,$ $\delta w_1\in F$ and for $i=2,3,\dots,n,$ $\delta w_i\in F[w_1,\dots, w_{i-1}]$. From \eqref{gii-defn}, it is clear that if $\eta$ is a generalized iterated integral of $F$ and $F\langle \eta\rangle$ has $F^\delta$ as its field of constants, then $F\langle \eta\rangle$ is a $\delta$-subfield of a Picard-Vessiot $\delta$-extension  of $F$ with a unipotent algebraic group as the differential Galois group.

\sloppy Let $E$ be a Picard-Vessiot $\delta$-extension  of $F.$ The Picard-Vessiot $F$-algebra of $E$ is the $F$-subalgebra of $E$ consisting of all solutions in $E$ of linear homogeneous differential equations over $F$ with respect to $\delta .$ This $F$-algebra is known to be finitely generated over $F$ and it is also a $\delta$-ring. Suppose that the differential Galois group of $E$ over $F$ is isomorphic to a unipotent algebraic group.
Then, there are $F$-algebraically independent elements $y_1,\dots,y_n\in E$ such that the Picard-Vessiot $F$-algebra of $E$ is the ring $F[y_1,\dots,y_n]$, where $\delta y_1\in F$ and for each $i=2,3,\dots, n,$ $\delta y_i\in F[y_1,\dots,y_{i-1}]$ (See \cite[Theorem 4.4]{KRS24}). Furthermore, any differential field $K$ intermediate to $F$ and $E$ is the field of fractions of the ring $K\cap F[y_1,\dots,y_n]$.

We shall now show that every element of the Picard-Vessiot $F$-algebra $F[y_1,\dots,y_n]$ is a generalized iterated integral of $F$. Fix a monomial ordering on the (polynomial) ring $F[y_1,\dots,y_n].$ For any element $y\in F[y_1,\dots,y_n],$ write $$y=a_ry^{r_1}_1y^{r_2}_2\dots y^{r_n}_n + \text{ lower order terms},$$ where $a_r\neq 0.$ As $\delta y_i\in F[y_1,\dots,y_{i-1}],$  we observe that $a^{-1}_r\delta y$ is of lower order. Repeating this process we obtain a natural number $r$ and nonzero elements $a_1,\dots, a_r$ such that    
$a_1\delta (a_2\delta (\dots(a_r\delta (y))))=0.$ Thus, elements of the Picard-Vessiot $F$-algebra of $E$ are generalized iterated integrals of $F.$ For more details, we refer the reader to \cite{SSS25}.

\section{elementary first $\delta$-integrals to elementary $\delta$-integrals}

In this section, we prove our main theorem. We will retain all the notations and terminologies from the earlier sections.

%%%%%%%%%%%%%%%%%%%%%%%%%%%%%%%%%%%%%%%%%%%%%%%%%%%%%%%%%%%%%%%%%%%%%%%%%%%%%%%%%%%%%%%%%%%%%%%%%%%%
\begin{lemma}\label{Induction-step3} Let $F(x_1,x_2)$ be a rational function field and a $\delta $-extension of $F$ such that $F(x_1,x_2)^{\delta }=
F^\delta$, $\delta x_2\in F[x_1]$ and $\delta x_1\in F$. Suppose that there are nonzero elements $v_1,\dots,v_r\in F(x_1,x_2)$, $\mathbb Q$-linearly independent constants $c_1,\dots,c_r$ and an element $z\in F(x_1,x_2)$ such that  $$\sum_{j=1}^{r}c_j\frac{\delta v_j}{v_j}+\delta z\in F[x_1,x_2].$$
Then, $v_1,\dots,v_r\in F$ and there is an element $z_0\in F(x_1)$ such that $$\sum_{j=1}^{r}c_j\frac{\delta v_j}{v_j}+\delta z_0\in F[x_1].$$
\end{lemma}

\begin{proof} Since $\delta  x_2\in F[x_1]$, from Lemma \ref{Rosenlicht-lemma}, we obtain that $v_1,\dots,v_r\in F(x_1)$ and that $z\in F(x_1)[x_2].$ Write \begin{equation}\label{induction-equation}\sum_{j=1}^{r}c_j\frac{\delta v_j}{v_j}+\delta z=w_0+w_1x_2+\dots+w_mx^m_2\in F[x_1][x_2].\end{equation}
Then, 
$$z=z_0+z_1x_2+\dots+z_{m}x^{m-1}_2+z_{m+1}x^{m+1}_2\in F(x_1)[x_2],$$ where\footnote{$z_{m+1}$ may be zero.} $z_{m+1}\in C$. Now, from Equation (\ref{induction-equation}),  we obtain the following set of equations by comparing coefficients: \begin{equation}\label{system-polynomialeqns}\tag{S}\begin{aligned}\begin{cases}\delta z_{m+1}&=0\in F[x_1]\\ \delta z_m+(m+1)z_{m+1}\delta x_2&=w_m\in F[x_1]\\ \vdots& \\ \delta z_i+(i+1)z_{i+1}\delta  x_2&=w_i\in F[x_1]\\ \vdots\\ \delta z_1+2z_{2}\delta  x_2&=w_1\in F[x_1]\\ \sum_{j=1}^{r}c_j\frac{\delta v_j}{v_j}+\delta z_0+z_1\delta x_2&=w_0\in F[x_1]\end{cases}\end{aligned}\end{equation} 
Since $z_{m+1}$ is a constant, we have $\delta z_m=w_m-(m+1)z_{m+1}\delta x_2\in F[x_1]$. Consider a Picard-Vessiot extension of $F$ that contains $F(x_1,x_2)$ and denote its Picard-Vessiot $F$-algebra by $R$.  Since $F(x_1)$ itself is a Picard-Vessiot extension of $F$ with $F[x_1]$ as its Picard-Vessiot $F$-algebra, it follows that $R\cap F(x_1)=F[x_1]$. Observe that $\delta z_m\in F[x_1]\subset R$, that is, $\delta z_m$ is a solution of a linear homogeneous differential operator over $F$. Therefore,  $z_m$ must also be a solution of a linear homogeneous differential operator over $F$ and thus, $z_m\in F[x_1]$. Now, since $\delta  z_{m-1}=w_{m-1}-mz_m\delta x_2$, we obtain that $\delta z_{m-1}\in F[x_1]$ and consequently that $z_{m-1}\in F[x_1]$. Recursively, using the set of equations \eqref{system-polynomialeqns}, we conclude that $z_1\in F[x_1]$ and  that $$\sum_{j=1}^{r}c_j\frac{\delta v_j}{v_j}+\delta z_0=w_0-z_1\delta x_2\in F[x_1].$$  Finally, we apply Lemma \ref{Rosenlicht-lemma} and obtain that $v_1,\dots,v_r\in F$. \end{proof}
%%%%%%%%%%%%%%%%%%%%%%%%%%%%%%%%%%%%%%%%%%%%%%%%%%%%%%%%%%%%%%%%%%
%%%%%%%%%%%%%%%%%%%%%%%%%%%%%%%%%%%%%%%%%%%%%%%%%%%%%
	\begin{proof}[Proof of Theorem \ref{intromain}] 
     Suppose that $E$ is an elementary $\Delta$-extension of $F(Y)$ such that $E^\Delta \subsetneqq E^\delta .$ Then, from \cite[Theorem]{MR704611}, there exist nonzero elements $v_1,\dots,v_r$ algebraic over $F(Y)$, an element $v$ algebraic over $F(Y)$ and constants $c_1\dots,c_r$ such that
		\begin{equation}\label{eqn-ineqn}
			\sum_{j=1}^{r}c_j\frac{\delta v_j}{v_j}+\delta v=0 \ \ \ \text{and} \ \ \ \sum_{j=1}^{r}c_j\frac{\delta_i v_j}{v_j}+\delta_i v\neq0
		\end{equation}
		for some $i$, $0\leq i\leq r$.  We shall further assume that the constants $c_1,\dots,c_r$ are linearly independent over the rational numbers $\mathbb Q$ (See \cite[p.158]{Ros68}).
We claim that the nonzero elements $v_1,\dots, v_r$, appearing in the above equation, can be assumed to be in $F$ and that $v$ not in $F$ but belongs to $F(Y)$. This is accomplished in three steps.

\emph{Step 1}: \emph{The elements $v_1,\dots, v_n,v$ can not be simultaneously algebraic over $F$}.

Suppose for a contradiction that  $v_1,\dots, v_r,v$ are algebraic over $F.$ Note that any derivation on $F$ extends uniquely to its algebraic closure. Since for all $f\in F$ and for all $i$, $1\leq i\leq r$,  we have $\delta_i f=0$, it follows that $\delta_i u=0$ for all $u\in \{v,v_1,\dots,v_r\}.$ Therefore $\sum_{j=1}^{r}c_j\frac{\delta_i v_j}{v_j}+ \delta_i v=0$ for all $i$, $1\leq i\leq r$. From the definition of $\delta $, we now have $\delta  u=\delta_0 u$ for all $u\in \{v,v_1,\dots,v_r\}.$ Therefore, $0=\sum_{j=1}^{r}c_j\frac{\delta v_j}{v_j}+\delta v$  $=\sum_{j=1}^{r}c_j\frac{\delta_0v_j}{v_j}+ \delta_0v=0.$ This contradicts the existence of a $\delta_i\in \Delta$ guaranteed by \eqref{eqn-ineqn}. Hence, the claim is proved.

Let $l\geq 1$ be the smallest integer such that there are elements $v, v_1,\dots, v_r$ and constants $c_1,\dots, c_r$ satisfying \eqref{eqn-ineqn} and that all the elements $v, v_1,\dots, v_r$ are algebraic over $F(y_1,\dots,y_{l})$. Let $F_0:=F$ and for $1\leq i\leq l$, $F_i=F_{i-1}(y_i)$.    

\emph{Step 2}: \emph{There are nonzero elements $v_1,\dots, v_r\in F_{l-1}$ and an element $v=cy_l+z$, where $c$ is a nonzero constant and $z\in F_{l-1}$, satisfying \eqref{eqn-ineqn}.}

We know that there are elements $v_1,\dots, v_r,v$ algebraic over $F_{l-1}(y_l)$ such that 
$$\sum_{j=1}^{r}c_j\frac{\delta v_j}{v_j}+\delta v=0.$$
Apply \cite[Theorem 2]{Ros76} to obtain that $v_1,\dots, v_r$ are algebraic over $F_{l-1}$ and that there is constant $c$ and an element $z$ algebraic over $F(y_1,\dots,y_{l-1})$ such that $v=cy_l+z$. For the choice of $l$, it is clear that $c\neq 0$. Let $N$ be finite Galois extension of $F(y_1,\dots, y_{l-1})$ containing $v_1,\dots,v_r,z$. Since $y_l$ is transcendental over $F_{l-1}$, the natural restriction map is an isomorphism between the Galois groups $\mathrm{Gal}(NF_l|F_l)$ and $\mathrm{Gal}(N|F_{l-1})$. Let $$\mathrm{Tr}(v):=\sum_{\sigma\in \mathrm {Gal}(NF_l|F_l)}\sigma (v)\, ,\qquad\qquad\mathrm{Nr}(v):=\prod_{\sigma\in \mathrm{Gal}(NF_l|F_l)}\sigma(v)$$ and $m$ be the order of the group $\mathrm{Gal}(N|F_{l-1})$. Then, for each $1\leq i\leq r$, $\mathrm{Tr}(v_i)\in F_{l-1}$, $\mathrm{Nr}(v_i)\in F_{l-1}$ and $$\mathrm{Tr}(v)=mcy_l+\mathrm{Tr}(z).$$ Since $c\neq 0$ and $\mathrm{Tr}(z)\in F_{l-1}$,  $\mathrm{Tr}(v)$ is a nonzero element of $F_l$ and that $\mathrm{Tr}(v)\notin F_{l-1}$. Observe that $$\sum^r_{i=1}c_i\frac{\delta \mathrm{Nr}(v_i)}{\mathrm{Nr}(v_i)}+\delta  \mathrm{Tr}(v)=0.$$
Since $\delta_l\mathrm{Tr}(v)=\delta_l(mcy_l)+\delta_l\mathrm{Tr}(z)=mc\neq 0$ and that $\delta_l\mathrm{Nr}(v_i)=0$, we also have $$\sum_{j=1}^{r}c_j\frac{\delta_l \mathrm{Nr}(v_j)}{v_j}+\delta_l \mathrm{Tr}(v)\neq0.$$ This completes the proof of \emph{Step 2}.

\emph{Step 3}: \emph{Elements $v_1,\dots, v_r$ in \eqref{eqn-ineqn} can be assumed to be in $F$.}

From \emph{Step 2}, we know that there are nonzero elements $v_1,\dots, v_r\in F_{l-1}$ and an element $v=cy_l+z$, where $c$ is some nonzero constant and $z\in F_{l-1}$, satisfying \eqref{eqn-ineqn}. We shall now show that $v_1,\dots,v_r$ indeed belong to $F$. Since $\delta z=\delta v-c\delta y_l$ and that $\delta y_l\in F[y_1,\dots,y_{l-1}]$, we have \begin{equation}\label{inductionstep}
			\sum_{j=1}^{r}c_j\frac{\delta v_j}{v_j}+\delta z\in F[y_{1},\dots,y_{l-2},y_{l-1}], 
		\end{equation}
where $v_1,\dots, v_r,z\in F_{l-1}$. Now, a repeated application of Lemma \ref{Induction-step3} will show that $v_1,\dots,v_r\in F$. This concludes the proof of \emph{Step 3} and hence the claim. 

Now, we shall return to proof of the Theorem. We now have an element $v\in F(Y)$, $v\notin F$, nonzero elements $v_1,\dots,v_r\in F$ and $\mathbb Q$-linearly independent constants $c_1,\dots, c_r$ such that  $$\delta v=\sum_{j=1}^{r}c_j\frac{\delta v_j}{v_j}.$$ 
This clearly implies that $v$ is an elementary $\delta $-integral of $F$ and concludes the proof of the implication part of the theorem.
    
Now, to prove the converse of the theorem, suppose that $F(Y)$ contains an elementary $\delta $-integral $\eta$ which is not in $F.$ Choose the smallest integer $m\geq 1$, constants $c_1,\dots,c_m$, nonzero elements $v_1,\dots, v_m\in F$ and an element $v\in F$ such that \begin{equation}\label{D-logexpression}\delta \eta=\sum^m_{i=1} c_i\frac{\delta v_i}{v_i}+\delta v.\end{equation} Note that $\delta v_i=\delta_0v_i$ and $\delta v=\delta_0 v.$ 
Consider the rational function field $F(Y)(z_1,\dots,z_m)$ of $F(Y)$ and extend the derivations of $\Delta$ using the rule \begin{equation}\label{definition-DEltaextn}\delta_i z_j= \frac{\delta_i v_j}{v_j}\,\, \text{for every } \delta_i\in \Delta\,\, \text{and for every } j,  1\leq j\leq m.\end{equation} Note that for each $v_j$ and for each $i=1,\dots, n,$ we have $\delta_i v_j=0$ as $v_j\in F$ and that $\delta  z_j=\delta v_j/v_j$. 

Suppose that $x\in F(Y)(z_1,\dots,z_m)$ is a $\Delta$-constant. Without loss of generality, we may assume that $F^\delta$ is the field of constants of the $\Delta$-field $F(Y)(z_1,\dots,z_{m-1}).$ Since for any $\delta_i\in \Delta$,  we have $\delta_i x=0$, it can be easily  seen (Use \cite[Proposition 2.2(b)]{KS19} for every $\delta_i\in \Delta$) that  there is an element $z\in F(Y)(z_1,\dots,z_{m-1})$  such that for every $\delta_i\in \Delta$, $$\delta_i z_m=\frac{\delta_i v_m}{v_m}=\delta_i z.$$ Then, from \cite[Section 2]{Kol68}, there are constants $e_1,\dots.e_{m-1}\in F^\delta$, not all zero, and an element $w\in F(Y)$ such that $$z=\sum^{m-1}_{i=1}e_iz_i+w.$$ Now, for any $\delta_i \in \Delta$, we have $$\frac{\delta_i v_m}{v_m}=\sum^{m-1}_{i=1}e_i\frac{\delta_i v_i}{v_i}+\delta_i w.$$ Therefore, from \eqref{D-logexpression}, we obtain $$\delta \eta=\sum^{m-1}_{i=1} (c_i+c_me_i)\frac{\delta v_i}{v_i}+\delta (v+c_mw).$$ This contradicts the choice of $m$. Thus, we have shown that $F(Y)(z_1,\dots,z_m)$ is a $\Delta$-elementary extension of $F(Y).$ Finally, we observe that  $\delta $ is also a derivation on $F(Y)(z_1,\dots,z_m)$ and that $$\delta z_i=\frac{\delta v_i}{v_i}.$$ Therefore, $$\delta \left(\eta-\sum^m_{i=1}c_iz_i- v\right)=0.$$
Thus, $\eta-\sum^m_{i=1}c_iz_i- v$ is an elementary first $\delta $-integral of the system \eqref{rationalsystem}. Since there is at least one of the constants $c_1,\dots,c_m$ is nonzero, $\eta-\sum^m_{i=1}c_iz_i- v\notin F$.\end{proof}

\begin{remark}\label{Singersresult}
    In \cite[Corollary 2]{MR704611}, it is shown that if for each $i=1,2,\dots, n$,  $\delta y_i=:u_i\in F$ and that the system $\delta y_1=u_1,\dots, \delta y_n=u_n$ has an elementary first $\delta$-integral then, there are constants $c_1,\dots, c_n, d_1,d_2,\dots,d_m$ and nonzero elements $v_1,\dots,v_m\in F$ and an element $v\in F$ such that  \begin{equation}\label{SingerCorollary}\sum^n_{i=1}c_iu_i=\sum^m_{j=1}d_j\frac{\delta v_j}{v_j}+\delta v.\end{equation}
    This result can be recovered from Theorem \ref{intromain}, in the case when $F(Y)^\delta=F^\delta$, as follows. From Theorem \ref{intromain}, we know that there is an elementary $\delta$-integral $u\in F(Y)$ of $F$ with $u\notin F$. That is, there are constants $d_1,\dots,d_m$, nonzero elements $v_1,\dots,v_m$ and an element $v\in F$ such that $\delta u=\sum^m_{j=1}d_j\delta v_j/v_j+ \delta v$. Apply the Kolchin-Ostrowki Theorem \cite[Section 2]{Kol68} to first obtain that $u=g+\sum^n_{i=1}c_iy_i$, for some constants $c_1,\dots,c_n$ and $g\in F$. Now, taking derivatives, we obtain an equation of form  \eqref{SingerCorollary}. 
\end{remark}

We conclude this section with the following proposition, which describes the elementary $\delta$-integrals for certain $\delta$-fields $F(Y)$. 

\begin{proposition}\label{iteratedintegrals-firstintegrals} Suppose that there is an element $x\in F$ such that $\delta x=1$ and that  there is a iterated integral $\eta\in F(Y)$ of $F$ with $\delta ^r\eta\in F,$ $\delta ^{r-1}\eta\notin F$, $F(Y)=F\langle \eta\rangle$ and $F(Y)^\delta =F^\delta$.
If there is an elementary $\Delta$-extension $E$ of $F(Y)$ such that  $E^\Delta \subsetneqq E^{\delta }$ then, there are polynomials $p_0,p_1,\dots,p_{r-1}\in F^\delta[x]$ with each $p_i$ having degree $\leq i$ such that $$\sum_{i=0}^{r-1}p_i\delta ^{i}\eta\notin F$$ is an elementary $\delta $-integral of $F$. 
	\end{proposition}

Before we proceed to state the proposition, we note the following: If $f_0,f_1,\dots,f_{n-1}\in F$, that is, $\delta y_i\in F$ for each $1\leq i\leq n$, $F(Y)^\delta=F^\delta$ and that $F$ has an element $x$ such that $\delta x=1$ then there is an iterated integral $\eta\in F(Y)$ of $F$ such that $F(Y)=F\langle \eta\rangle$ (See \cite[Theorem 2.3]{SSS25}). Thus, the above proposition is applicable to all such integral extensions of $F$.

\begin{proof}[Proof of Proposition \ref{iteratedintegrals-firstintegrals}] 
By Lemma \ref{integral_extension}, $F\langle \eta\rangle =F(w_1,\dots,w_r)$ where for each $j=0,1,\dots, r-1,$ \begin{align}\label{expression_of_integrals}
         w_{j+1}&= \sum^j_{i=0}(-1)^i\frac{x^i}{i!}\delta ^{r-(j+1-i)}\eta\notag \\ &= \sum^{r-1}_{i=r-(j+1)}(-1)^{i+j+1-r}\frac{x^{i+j+1-r}}{(i+j+1-r)!}\delta ^{i}\eta . \end{align}
By Theorem \ref{intromain}, there exists an elementary $\delta $-integral $u\in F\langle \eta\rangle$ of $F$ such that $u\notin F$. Using  \cite[Section 2]{Kol68}, we find  constants $c_1,\dots, c_r$ and an element $u_0\in F$ such that \begin{equation}\label{Kol-Ost}\tilde{u}=\sum^r_{j=1}c_jw_j\, ,\end{equation} where $\tilde{u}=u-u_0$. Note that $\tilde{u}$ is also an elementary $\delta $-integral of $F$. Now, from Equation \eqref{expression_of_integrals}, we immediately obtain \begin{equation}\label{elemint-intermsofderivatives}\tilde{u}=\sum_{i=0}^{r-1}p_i\delta ^i\eta,\end{equation} where for each $i=0,1,\dots,r-1,$ the polynomial $p_i\in F^\delta[x]$ is obtained by gathering the coefficients of $\delta ^i\eta.$ Finally, in Equation \eqref{expression_of_integrals}, $i+j+1-r\leq i$ as $j\leq r-1.$ Therefore, each polynomial $p_i$ is of degree at most $i.$\end{proof}

%%%%%%%%%%%%%%%%%%%%%%%%%%%%%%%%%%%%%%%%%%%%%%%%%%%%%%%%
%%%%%%%%%%%%%%%%%%%%%%%%%%%%%%%%%%%%%%%%%%%%%%%%%%%%%%%%
%%%%%%%%%%%%%%%%%%%%%%%%%%%%%%%%%%%%%%%%%%%%%%%%%%%%%%%%
%%%%%%%%%%%%%%%%%%%%%%%%%%%%%%%%%%%%%%%%%%%%%%%%%%%%%%%%%%%%%%%%%%%%%%%%%%%%%%%%%%%%%%%%%%%%%%%%%%%%%%%%%%%%%%%%
%%%%%%%%%%%%%%%%%%%%%%%%%%%%%%%%%%%%%%%%%%%%%%%%%%%%%%%%
%%%%%%%%%%%%%%%%%%%%%%%%%%%%%%%%%%%%%%%%%%%%%%%%%%%%%%%%
%%%%%%%%%%%%%%%%%%%%%%%%%%%%%%%%%%%%%%%%%%%%%%%%%%%%%%%%
\section{Applications}
%%%%%%%%%%%%%%%%%%%%%%%%%%%%%%%%%%%%%%%%%%%%%%%%%%%%%%%%
%%%%%%%%%%%%%%%%%%%%%%%%%%%%%%%%%%%%%%%%%%%%%%%%%%%%%%%%
%%%%%%%%%%%%%%%%%%%%%%%%%%%%%%%%%%%%%%%%%%%%%%%%%%%%%%%%
%%%%%%%%%%%%%%%%%%%%%%%%%%%%%%%%%%%%%%%%%%%%%%%%%%%%%%%%%%%%%%%%%%%%%%%%%%%%%%%%%%%%%%%%%%%%%%%%%%%%%%%%%%%%%%%%
%%%%%%%%%%%%%%%%%%%%%%%%%%%%%%%%%%%%%%%%%%%%%%%%%%%%%%%%
%%%%%%%%%%%%%%%%%%%%%%%%%%%%%%%%%%%%%%%%%%%%%%%%%%%%%%%%
%%%%%%%%%%%%%%%%%%%%%%%%%%%%%%%%%%%%%%%%%%%%%%%%%%%%%%%%
In this section, we provide a $\delta$-field $F$ and a $\delta$-extension $F(Y)$ which does not contain any elementary $\delta$-integrals. This in turn implies, thanks to Theorem \ref{intromain}, that the system has no elementary first $\delta$-integrals. 

Consider the rational function field $\mathbb C(x)$ over the field of complex numbers $\mathbb C$ with the derivation $\delta :=\frac{d}{dx}$. Let $h:=(x-a_1)(x-a_2)(x-a_3)$,
where $a_1,a_2,a_3$ are distinct complex numbers and $z$ be an algebraic element over $\mathbb C(x)$ satisfying $z^2 =h$. Let $F=\mathbb C(x,z).$ It is well-known that the rational field $F(\eta)$ with the derivation $\delta \eta=1/z$ has $\mathbb C$ as its field of constants and that the element $\eta$ is called an \emph{elliptic integral}. Furthermore, $\eta$ is not an elementary $\delta$-integral of $F$, that is, $\eta$ does not belong to any $\delta$-elementary extension of $F$ (See \cite[p.36]{Rit48}).

\begin{proposition}\label{elliptic-intelliptic}
For every rational function $g\in \mathbb C(x)$ having at least one simple pole, the system \begin{equation}\label{system-elliptic}
		\tag{E}
	\begin{cases}	\begin{aligned} 
			\delta \zeta&=g\eta\\
    \delta \eta&=\frac{1}{z}
		\end{aligned}\end{cases}
	\end{equation}
 has no rational first $\delta$-integral  and no elementary first $\delta$-integrals. 
 \end{proposition}

\begin{proof} 
We first consider the rational function field $F(\eta, \zeta)$ with the derivation $\delta $ induced by the system \eqref{system-elliptic}. We shall first show that the system has no rational first $\delta$-integrals, that is, $F(\eta,\zeta)^{\delta }=\mathbb C.$ We have already observed that $F(\eta)^{\delta }=\mathbb C$.

Suppose that  there is a constant in $F(\eta,\zeta)$ which is not in $\mathbb C.$ Then from \cite[Proposition 2.2(b)]{KS19}, it follows that there is a $y\in F(\eta)$ such that $\delta y=g\eta.$
Apply Lemma \ref{Rosenlicht-lemma} and conclude that $y=\alpha_2\eta^2+\alpha_1\eta+\alpha_0\in F[\eta]$, $\alpha_2\in \mathbb C$, $\alpha_1, \alpha_0\in F$. Comparing the coefficients of $\eta$ in $\delta y=g\eta$, we have \begin{equation}\label{weakorthogonality-1}\frac{2\alpha_2}{z}+\delta \alpha_1=g\end{equation} 

Therefore, $\delta (2\alpha_2 \eta+\alpha_1)=g$. Since $g\in \mathbb C(x)$ has a simple pole, $2\alpha_2 \eta+\alpha_1$ must belong to a transcendental $\delta$-elementary extension of $F$. Thus, $\alpha_2$ must be nonzero and that $\eta$ is also an elementary $\delta$-integral of $F$, a contradiction. Hence $F(\eta,\zeta)^{\delta }=\mathbb C$.

Now we shall show that $F(\eta,\zeta)$ contains no elementary $\delta$-integrals of $F$. Suppose for a contradiction that there is an element $y\in F(\eta, \zeta)$ and $y\notin F$ is an elementary $\delta$-integral of $F$. Since $\eta$ is not an elementary $\delta$-integral of $F$, we have $y\notin F(\eta)$. Note that $F(\eta,y)=$ $F(\eta, \zeta)$ (See \cite[Corollary 2.1.1]{srin-2010-IAE}).  Therefore, from \cite[Section 2]{Kol68}, there is a nonzero constant $c$ such that $c\zeta+y\in F(\eta).$ Since $\delta (c\zeta+y)=cg\eta+\delta y\in F[\eta],$ by Lemma \ref{Rosenlicht-lemma} 
$$c\zeta+y=\beta_0+\beta_1\eta+\beta_2\eta^2 $$
for some $\beta_2\in \mathbb C$ and $\beta_0,\beta_1\in F.$ Taking derivatives, 
$$cg\eta+\delta y=\delta \beta_0+\frac{\beta_1}{z}+\left(\delta \beta_1+\frac{2\beta_2}{z} \right)\eta.$$
Therefore, $$g=\delta \left(
\frac{\beta_1}{c}\right)+\frac{2\beta_2}{cz}=\delta \left(\frac{\beta_1+2\beta_2\eta}{c}\right).$$ But, as noted earlier in the paragraph succeeding \eqref{weakorthogonality-1}, no such $\beta_1, \beta_2$ exist in $F$ and we have arrived at a contradiction. 
This completes the proof of the proposition.\end{proof}

%%%%%%%%%%%%%%%%%%%%%%%%%%%%%%%%%%%%%%%%%%%%%%%%%%%%%%%%
%%%%%%%%%%%%%%%%%%%%%%%%%%%%%%%%%%%%%%%%%%%%%%%%%%%%%%%%
%%%%%%%%%%%%%%%%%%%%%%%%%%%%%%%%%%%%%%%%%%%%%%%%%%%%%%%%
%%%%%%%%%%%%%%%%%%%%%%%%%%%%%%%%%%%%%%%%%%%%%%%%%%%%%%%%

    \bibliographystyle{unsrt}
	\bibliography{SS}
	
\end{document}